\documentclass[12pt]{amsart}

\usepackage{amsfonts,amsmath,amsthm,amssymb,graphicx,tikz}

\tikzstyle{vertex}=[circle, draw, inner sep=0pt, minimum size=6pt]
\newcommand{\vertex}{\node[vertex]}

\theoremstyle{plain}
\newtheorem{thm}{Theorem}[section]

\newtheorem{prop}[thm]{Proposition}

\theoremstyle{definition}

\newtheorem{ques}[thm]{Question}

\newtheorem{obs}[thm]{Observation}
\newtheorem{prob}[thm]{Problem}

\theoremstyle{remark}
\newtheorem{rem}[thm]{Remark}

\newcommand{\ov}{\overline}
\newcommand{\g}{\Gamma}
\newcommand{\bb}{{\bf b}}
\newcommand{\bc}{{\bf c}}

\newcommand{\M}{{\mathcal M}}
\newcommand{\F}{{\mathcal F}}

\begin{document}


\title{On sign-symmetric signed graphs}

\author[E. Ghorbani, W.H. Haemers, H.R. Maimani and L. Parsaei Majd]{Ebrahim Ghorbani}

\address{E. Ghorbani, Department of Mathematics, K. N. Toosi University of Technology,
 P.O. Box 16765-3381, Tehran, Iran.}
\email{ghorbani@kntu.ac.ir}

\author[]{Willem H. Haemers}
\address{W.H. Haemers, Department of Econometrics and Operations Research,
Tilburg University, Tilburg, The Netherlands.}
\email{haemers@uvt.nl}

\author[]{Hamid Reza Maimani}

\address{H.R. Maimani, Mathematics Section, Department of Basic Sciences,
Shahid Rajaee Teacher Training University, P.O. Box 16785-163, Tehran,
Iran, Iran.}

\email{maimani@ipm.ir}

\author[]{Leila Parsaei Majd}

\address{L. Parsaei Majd, Mathematics Section, Department of Basic Sciences,
Shahid Rajaee Teacher Training University, P.O. Box 16785-163, Tehran,
Iran.}

\email{leila.parsaei84@yahoo.com}

\begin{abstract}
A signed graph is said to be  sign-symmetric if it is switching isomorphic to its negation.
Bipartite signed graphs are trivially sign-symmetric.
We give new constructions of non-bipartite sign-symmetric signed graphs.
Sign-symmetric signed graphs have a symmetric spectrum but not the other way around.
We present  constructions of signed graphs with symmetric spectra which are not sign-symmetric. This, in particular answers a problem posed by  Belardo, Cioab\u{a},  Koolen, and Wang (2018).

\end{abstract}






\thanks{}


\maketitle


\section{Introduction}
Let $G$ be a graph with vertex set $V$ and edge set $E$.
All graphs considered in this paper are undirected, finite, and simple (without loops or multiple edges).

A {\em signed graph} is a graph in which every edge has been declared positive or negative. In fact, a signed graph $\g$ is a pair $(G, \sigma)$, where $G=(V, E)$ is a graph, called the underlying graph, and $\sigma : E \rightarrow \{-1, +1\}$ is the sign function
or signature. Often, we write $\g=(G, \sigma)$ to mean that the underlying graph is $G$. 
The signed graph $(G,-\sigma)=-\Gamma$ is called the {\em negation} of $\Gamma$. 
Note that if we consider a signed graph with all edges positive, we obtain an unsigned graph.

Let $v$ be a vertex of a signed graph $\g$. {\em Switching} at $v$ is changing the
signature of each edge incident with $v$ to the opposite one. Let $X\subseteq V$.
Switching a vertex set $X$ means reversing the signs of all edges between $X$
and its complement. Switching a set $X$ has the
same effect as switching all the vertices in $X$, one after another.

Two signed graphs $\g=(G, \sigma)$ and $\g' = (G, \sigma')$ are said to be {\em switching equivalent} if there is a series of switching that transforms $\g$ into $\g'$. If $\g'$ is isomorphic to a switching of $\g$, we say that $\g$ and $\g'$ are {\em switching isomorphic} and we write $\g\simeq \g'$.
The signed graph $-\g$ is obtained from $\g$ by reversing the sign of all edges. A signed graph $\g=(G, \sigma)$ is said to be {\em sign-symmetric} if $\g$ is switching isomorphic to $(G,-\sigma)$, that is: $\Gamma \simeq -\Gamma$.

For a signed graph $\g=(G, \sigma)$, the adjacency matrix $A=A(\g)=(a_{ij})$ is an $n\times n$ matrix in which $a_{ij} = \sigma(v_i v_j)$ if $v_i$ and $\upsilon_j$ are adjacent, and $0$ if they are not. Thus $A$ is a symmetric matrix with entries $0, \pm 1$ and zero diagonal, and conversely, any such matrix is the adjacency matrix of a signed graph.
The spectrum of $\g$ is the list of eigenvalues of its adjacency matrix with their multiplicities.
 We say that  $\g$ has a {\em symmetric spectrum} (with respect to the origin) if for each eigenvalue $\lambda$ of $\g$, $-\lambda$ is also an eigenvalues of $\g$ with the same multiplicity.\\

Recall that (see \cite{brou-hamer}), the \textit{Seidel adjacency matrix} of a graph $G$ with the adjacency matrix $A$ is the matrix
$S$ defined by
\[
S_{uv}= \left\{
\begin{array}{rl}
0&  \text{if} ~u=v\\
-1 & \text{if}~u\sim v\\
1 &\text{if}~u\nsim v
\end{array}\right.
\]
so that $S = J - I - 2A$. The Seidel adjacency spectrum of a graph is the spectrum of its
Seidel adjacency matrix. If $G$ is a graph of order $n$, then the Seidel matrix of $G$ is the adjacency matrix of a signed complete graph $\Gamma$ of order $n$ where the edges of $G$ are precisely the negative edges of $\Gamma$. 

\begin{prop}
Suppose $S$ is a Seidel adjacency matrix of order $n$.
If $n$ is even, then $S$ is nonsingular, and if $n$ is odd, $\mathrm{rank}(S)\geq n-1$.
In particular, if $n$ is odd, and $S$ has a symmetric spectrum, then $S$ has an eigenvalue $0$ of multiplicity 1.
\end{prop}
\begin{proof}
We have $\det(S) \equiv \det(I-J) (\mathrm{mod}~2)$, and $\det(I-J)=1-n$.
Hence, if $n$ is even, $\det(S)$ is odd. So, $S$ is nonsingular.
Now, if $n$ is odd, any principal submatrix of order $n-1$ is nonsingular.
Therefore, $\mathrm{rank}(S)\geq n-1$.
\end{proof}

The goal of this paper is to study sign-symmetric signed graphs as well as signed graphs with symmetric spectra.
It is known that bipartite signed graphs are sign-symmetric. We give new constructions of non-bipartite sign-symmetric graphs.
It is obvious that sign-symmetric graphs have a symmetric spectrum but not the other way around (see Remark~\ref{rem:excep8} below).
We present  constructions of graphs with symmetric spectra which are not sign-symmetric. This, in particular answers a problem posed in \cite{belardo2019}.

\section{Constructions of sign-symmetric graphs}



We note that the property that two signed graphs $\g$ and $\g'$ are switching isomorphic is equivalent to the existence of a `signed' permutation matrix $P$ such that $PA(\g)P^{-1}=A(\g')$. If $\g$ is a bipartite signed graph, then we may write its adjacency matrix as
$$A =\begin{bmatrix}
O & B \\B^\top & O
\end{bmatrix}.$$
It follows that $PAP^{-1}=-A$ for
$$P =\begin{bmatrix}
-I & O \\O & I
\end{bmatrix},$$
which means that bipartite graphs are `trivially' sign-symmetric. So it is natural to look for non-bipartite sign-symmetric graphs.
The first construction was given in \cite{a.m.p} as follows.
\begin{thm}\label{sym.thm}
Let $n$ be an even positive integer and $V_1$ and $V_2$ be two disjoint sets of size $n/2$. 
Let $G$ be an arbitrary graph with the vertex set $V_1$. Construct the complement of $G$, that is $G^{c}$, with the vertex set $V_2$. Assume that $\Gamma=(K_{n}, \sigma)$ is a signed complete graph in which $E(G)\cup E(G^{c})$ is the set of negative edges. Then the spectrum of $\Gamma$ is sign-symmetric.
\end{thm}

Theorem \ref{sym.thm} says that for an even positive integer $n$, let $B$ be the adjacency matrix of an arbitrary graph on $n/2$ vertices. Then, the complete signed graph in which the negatives edges induce the disjoint union of $G$ and its complement, is sign-symmetric.

\subsection{Constructions for general signed graphs}

Let $\M_{r,s}$ denote the set of $r\times s$ matrices with entries from $\{-1,0,1\}$.
We give another construction generalizing the one given in Theorem~\ref{sym.thm}:

\begin{thm}\label{sym.thm2}
Let  $B,C\in\M_{k,k}$  be symmetric matrices where $B$ has a zero diagonal. 
Then the signed graph with the adjacency matrices
$$A =\begin{bmatrix}
B & C \\C & -B
\end{bmatrix}$$
is sign-symmetric on $2k$ vertices. 
\end{thm}
\begin{proof}
$$\begin{bmatrix}O & -I \\I & O\end{bmatrix}
\begin{bmatrix}B & C \\C & -B\end{bmatrix}
\begin{bmatrix}O & I \\-I & O\end{bmatrix}=
\begin{bmatrix}
-B & -C \\-C & B
\end{bmatrix}=-A$$
\end{proof}


Note that Theorem \ref{sym.thm2} shows that there exists a sign-symmetric graph for every even order.  

We define the family $\F$ of signed graphs as those which have an adjacency matrix satisfying the conditions given in Theorem~\ref{sym.thm2}. 
To get an impression on what the role of $\F$ is in the family of sign-symmetric graphs, we investigate small complete signed graphs.
All but one complete signed graphs with symmetric spectra of orders $4, 6, 8$ are illustrated in Fig.~\ref{sign4,6,8} (we show one signed graph in the switching class of the signed complete graphs induced by the negative edges). 
There is only one sign-symmetric complete signed graph of order $4$. There are four complete signed graphs with symmetric spectrum of order $6$, all of which are sign-symmetric, and twenty-one complete signed graphs with symmetric spectrum of order $8$, all except the last one are sign-symmetric, and together with the negation of the last signed graph, Fig.~\ref{sign4,6,8} gives all complete signed
graphs with symmetric spectrum of order $4$, $6$ and $8$. Interestingly, all of the above sign-symmetric signed graphs belong to $\F$.

The following proposition shows that $\F$ is closed under switching.


\begin{prop}
If $\g\in\F$ and $\g'$ is obtained from $\g$ by switching, then $\g'\in\F$.
\end{prop}
\begin{proof} Let $\g\in\F$. It is enough to show that if $\g'$ is obtained from $\g$ by switching with respect to its first vertex, then $\g'\in\F$.
We may write the adjacency matrix of $\g$ as follows:
$$A=\left[\begin{array}{c|cccccc|c|cccc}
0&&&&\bb^\top&&&c&&&\bc^\top&\\ 
&&&&&&&&&&&\\\hline
&&&&&&&&&&&\\
&&&&&&&&&&&\\
\bb&&&&B'&&&\bc&&&C'& \\
&&&&&&&&&&&\\
&&&&&&&&&&&\\\hline
&&&&&&&&&&&\\
c&&&&\bc^\top&&&0&&&-\bb^\top&\\ 
&&&&&&&&&&&\\\hline
&&&&&&&&&&&\\
&&&&&&&&&&&\\
\bc&&&&C'&&&-\bb&&&-B'&\\
&&&&&&&&&&&
\end{array}\right].$$
After switching with respect to the first vertex of $\g$, the adjacency matrix of the resulting signed graph is
$$\left[\begin{array}{c|cccccc|c|cccc}
0&&&&-\bb^\top&&&-c&&&-\bc^\top&\\ 
&&&&&&&&&&&\\\hline
&&&&&&&&&&&\\
&&&&&&&&&&&\\
-\bb&&&&B'&&&\bc&&&C'& \\
&&&&&&&&&&&\\
&&&&&&&&&&&\\ \hline
&&&&&&&&&&&\\
-c&&&&\bc^\top&&&0&&&-\bb^\top&\\ 
&&&&&&&&&&&\\\hline
&&&&&&&&&&&\\
&&&&&&&&&&&\\
-\bc&&&&C'&&&-\bb&&&-B'&\\
&&&&&&&&&&&
\end{array}\right].$$
Now by interchange the 1st and $(k+1)$-th rows and columns we obtain
\vspace{0.2 cm}
 $$\left[\begin{array}{c|cccccc|c|cccc}
0&&&&\bc^\top&&&-c&&&-\bb^\top&\\ 
&&&&&&&&&&&\\\hline
&&&&&&&&&&&\\
&&&&&&&&&&&\\
\bc&&&&B'&&&-\bb&&&C'& \\
&&&&&&&&&&&\\
&&&&&&&&&&&\\\hline
&&&&&&&&&&&\\
-c&&&&-\bb^\top&&&0&&&-\bc^\top&\\ 
&&&&&&&&&&&\\\hline
&&&&&&&&&&&\\
&&&&&&&&&&&\\
-\bb&&&&C'&&&-\bc&&&-B'&\\
&&&&&&&&&&&
\end{array}\right]$$
\\
which is a matrix of the form given in Theorem~\ref{sym.thm2} and thus $\g'$ is isomorphic with a signed graph in $\F$.
\end{proof}

In the following we present two constructions for complete sign-symmetric signed graphs using self-complementary graphs.

\subsection{Constructions for complete signed graphs}

In the following, the meaning of a self-complementary graph is the same as defined for unsigned graphs. 
Let $G$ be a self-complementary graph so that there is a permutation matrix $P$ such that $PA(G)P^{-1}=A(\ov G)$ and  $PA(\ov G)P^{-1}=A(G)$. It follows that
if $\g$ is a complete signed graph with $E(G)$ being its negative edges, then $A(\g)=A(\ov G)-A(G)$, (in other words, $A(\Gamma)$ is the Seidel matrix of $G$). It follows that
$PA(\g)P^{-1}=-A(\g)$. So we obtain the following:

\begin{obs}\label{obs:self} If $\g$ is a complete signed graph whose negative edges induce a self-complementary graph, then $\g$ is sign-symmetric.
\end{obs}

We give one more construction of sign-symmetric signed graphs based on self-complementary graphs as a corollary to Observation~\ref{obs:self}. We remark that a self-complementary graph of order $n$ exists whenever $n\equiv 0~\text{or}~1~(\mathrm{mod}~4)$.

\begin{prop} \label{cor:oddcase}
Let $G, H$ be two self-complementary graphs, and let $\g$ be a complete signed graph whose negative edges induce the join of $G$ and $H$ (or the disjoint union of $G$ and $H$).
Then $\g$ is sign symmetric. In particular, if $G$ has $n$ vertices, and if $H$ is a singleton, then the complete signed graph $\g$ of order $n+1$ with negative edges equal to $E(G)$ is sign-symmetric.
\end{prop}

In the following remark we present a sign-symmetric construction for non-complete signed graphs.

\begin{rem} Let $\g', \g''$ be two signed graphs which are isomorphic to $-\g', -\g''$, respectively. Consider the signed graph $\g$ obtained from joining $\g'$ and $\g''$ whose negative edges are the union of negative edges in $\g'$ and $\g''$. Then, $\g$ is sign-symmetric.
\end{rem}

\begin{rem}
By Proposition~\ref{cor:oddcase}, we have a construction of sign-symmetric complete signed graphs of order $n\equiv 0, 1~\text{or}~2~(\mathrm{mod}~4)$.
 All complete sign-symmetric signed graphs of order $5$ and $9$ (depicted in Fig.~\ref{sign5,9}) can be obtained in this way.
 There is just one sign-symmetric signed graph of order $5$ which is obtained by joining a vertex to a complete signed graph of order $4$ whose negative edges form a path of length $3$ (which is self-complementary).
 Moreover, there exist sixteen complete signed graphs of order $9$ with symmetric spectrum of which ten are sign-symmetric; the first three are not sign-symmetric, and when we include their negations we get them all. All of these ten complete sign-symmetric signed graphs can be obtained by joining a vertex to a complete signed graph of order $8$ whose negative edges induce a self-complementary graph. Note that there are exactly ten self-complementary graphs of order $8$. 
\end{rem}

\begin{thm}
There exists a complete sign-symmetric signed graph of order $n$ if and only if $n\equiv 0, 1~\text{or}~2~(\mathrm{mod}~4)$.
\end{thm}

\begin{proof}
Using the previous results obviously one can construct a sign-symmetric signed graph of order $n$ whenever $n\equiv 0, 1~\text{or}~2~(\mathrm{mod}~4)$. Now, suppose that there is a complete sign-symmetric signed graph $\Gamma$ of order $n$ with $n\equiv 3~(\mathrm{mod}~4)$. By \cite[Corollary 3.6]{detsiedel}, the determinant of the Seidel matrix of $\Gamma$ is congruent to $1-n~(\mathrm{mod}~4)$. Since $n\equiv 3~(\mathrm{mod}~4)$, the determinant of the Seidel matrix (obtained from the negative edges of $\Gamma$) is not zero. Hence, we can conclude that all eigenvalues of $\Gamma$ are non-zero. Therefore, $\Gamma$ cannot have a symmetric spectrum, and also it cannot be sign-symmetric.
\end{proof}

In \cite{LS} all switching classes of Seidel matrices of order at most seven are given. There is a error in the spectrum of one of the graphs on six vertices in {\cite[Table 4.1]{LS}} (2.37 should be 2.24), except for that, the results in \cite{LS} coincide with ours.

\section{Positive and negative cycles}
A graph whose connected components are $K_2$ or cycles is called an  {\em elementary graph}.
Like unsigned graphs, the coefficients of the characteristic polynomial of the adjacency matrix of a signed graph $\g$ can be described in terms of elementary subgraphs of $\g$.
\begin{thm}[{\cite[Theorem 2.3]{belardo2}}]\label{blardo} Let $\g=(G, \sigma)$ be a signed graph and
\begin{equation}\label{eq:characteristic}
P_{\g}(x)=x^n+a_1x^{n-1}+ \cdots +a_{n-1}x+a_n
\end{equation}
be the characteristic polynomial of the adjacency matrix of $\g$. Then
$$a_i=\underset{B\in \mathcal{B}_i}\sum {(-1)^{p(B)}2^{\vert c(B)\vert}\sigma(B)},$$
where $\mathcal{B}_i$ is the set of elementary subgraphs of $G$ on $i$ vertices, $p(B)$ is the number of components of $B$, $c(B)$ the set of cycles in $B$, and $\sigma(B)=\prod_{C\in c(B)}\sigma(C)$.
\end{thm}

\begin{rem}\label{rem:sym}
It is clear that $\g$ has a symmetric spectrum if and only if in its characteristic polynomial \eqref{eq:characteristic}, we have
 $a_{2k+1}=0$, for $k=1, 2, \ldots$.
\end{rem}

In a signed graph, a cycle is called {\em positive} or {\em negative} if the product of
the signs of its edges is positive or negative, respectively.
We denote the number of positive and negative $\ell$-cycles by $c_{\ell}^{+}$ and $c_{\ell}^{-}$, respectively. 

\begin{obs}\label{obs:sign-sym} For sign-symmetric signed graph, we have
$$c_{2k+1}^{+}=c_{2k+1}^{-}~\text{for}~k=1, 2, \ldots .$$
\end{obs}

\begin{rem}\label{rem:sign-symm} If in a signed graph $\g$, $c_{2k+1}^{+}=c_{2k+1}^{-}$ for all $k=1, 2, \ldots$, then it is not necessary that $\g$ is sign-symmetric. See the complete signed graph given in Fig.~\ref{fig:excep8}. For this complete signed graph we have $c_{2k+1}^{+}=c_{2k+1}^{-}$ for all $k=1, 2, \ldots$, but it is not sign-symmetric. Moreover, one can find other examples among complete and non-complete signed graphs. For example, the signed graph given in Fig.~\ref{fig:non-sign} is a non-complete signed graph with the property that $c_{2k+1}^{+}=c_{2k+1}^{-}$ for all $k=1, 2, \ldots$, but it is not sign-symmetric.
\end{rem}

By Theorem \ref{blardo}, we have that $a_3=2(c_{3}^{-}-c_{3}^{+})$.
By Theorem \ref{blardo} and Remark~\ref{rem:sym} for signed graphs having symmetric spectrum, we have
$c_{3}^{+}=c_{3}^{-}$. Further, for each complete signed graph with a symmetric spectrum, it can be seen that $c_{5}^{+}=c_{5}^{-}$.
However, the equality $c_{2k+1}^{+}=c_{2k+1}^{-}$ does not necessarily hold for $k\ge3$.
The complete signed graph in Fig.~\ref{fig:excep9} has a symmetric spectrum for which $c_{7}^{+} \neq c_{7}^{-}$.
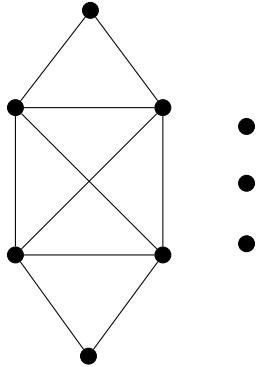
\begin{figure}[!htb]
\centering
\begin{tikzpicture}[scale=.7]
\vertex [fill] (a1) at (-5.3755369033828, 5.848814013183)[]{};
\vertex [fill] (a2) at (-4, 4)[]{};
\vertex [fill] (a3) at (-6.8, 4)[]{};
\vertex [fill] (b1) at (-6.8, 1.2)[]{};
\vertex [fill] (b2) at (-4, 1.2)[]{};
\vertex [fill] (b3) at (-5.4129901347276, -0.7230931414287)[]{};
\vertex [fill] (c1) at (-2.4107731585189, 3.6424782030517)[]{};
\vertex [fill] (c2) at (-2.4107731585189, 2.562292962675)[]{};
\vertex [fill] (c2) at (-2.4107731585189, 1.4131597282316)[]{};
\path
(a1) edge (a2)
(a1) edge (a3)
(a2) edge (a3)
(a2) edge (b1)
(a2) edge (b2)
(a3) edge (b1)
(a3) edge (b2)
(b1) edge (b2)
(b1) edge (b3)
(b2) edge (b3);
\end{tikzpicture}
\caption{The graph induced by negative edges of a complete signed graph on $9$ vertices with a symmetric spectrum but $c_{7}^{+} \neq c_{7}^{-}$}\label{fig:excep9}
\end{figure}

\begin{rem}
There are some examples showing that for a non-complete signed graph we have $c_{2k+1}^{+}=c_{2k+1}^{-}$ for all $k=1, 2, \ldots$, but their spectra are not symmetric. As an example see Fig.~\ref{fig:non-sign}, (dashed edges are negative; solid edges are positive).
\end{rem}

Now, we may ask a weaker version of the result mentioned in Remark~\ref{rem:sign-symm} as follows.

\begin{ques}\label{ques2}
Is it true that if in a complete signed graph $\g$, $c_{2k+1}^{+}=c_{2k+1}^{-}$ for all $k=1, 2, \ldots$, then $\g$ has a symmetric spectrum?
\end{ques}
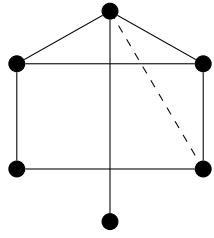
\begin{figure}[h!]
\centering
\begin{tikzpicture}[scale=.7]
\vertex [fill] (a1) at (0,2) []{};
\vertex [fill](a2) at (0,0) []{};
\vertex [fill](a3) at (1.77,-1) []{};
\vertex [fill](a4) at (3.54,0) []{};
\vertex [fill](a5) at (3.54,2) []{};
\vertex [fill](a6) at (1.77,3) []{};
\draw [dashed] (3.54,0) -- (1.77,3);
\path
(a1) edge (a2)
(a1) edge (a5)
(a1) edge (a6)
(a2) edge (a4)
(a3) edge (a6)
(a4) edge (a5)
(a5) edge (a6);
\end{tikzpicture}
 \caption{A signed graph with $c_{2k+1}^{+}=c_{2k+1}^{-}$ for $k=1, 2, \ldots$, but its spectrum is not symmetric}\label{fig:non-sign}
\end{figure}


\section{Sign-symmetric vs. symmetric spectrum}

\begin{rem}\label{rem:excep8}
Consider the complete signed graph whose negative edges induces the graph of Fig.~\ref{fig:excep8}.
This graph has a symmetric spectrum, but it is not sign-symmetric. Note that this complete signed graph has the minimum order with this property. Moreover, for this complete signed graph we have the equalities $c_{2k+1}^{+}=c_{2k+1}^{-}$ for $k=1, 2, 3$.
\end{rem}
\begin{figure}[!htb]
\centering
\begin{tikzpicture}[scale=.7]
\vertex [fill] (a1) at (0,0)[]{};
\vertex [fill] (a2) at (5,0)[]{};
\vertex [fill] (a3) at (2.5,4.33)[]{};
\vertex [fill] (b1) at (1.253,2.17)[]{};
\vertex [fill] (b2) at (2.5,0)[]{};
\vertex [fill] (b3) at (3.782,2.15)[]{};
\vertex [fill] (c1) at (4.42,3.73)[]{};
\vertex [fill] (c2) at (5,3)[]{};
\path
(a1) edge (a2)
(a2) edge (a3)
(a3) edge (a1)
(b1) edge (b2)
(b2) edge (b3)
(b3) edge (b1);
\end{tikzpicture}
\caption{The graph induced by negative edges of a complete signed graph on $8$ vertices with a symmetric spectrum but not sign-symmetric}\label{fig:excep8}
\end{figure}
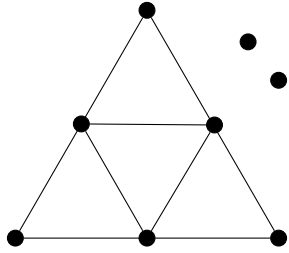

\begin{rem}
A {\em conference matrix} $C$ of order $n$ is an $n\times n$ matrix with zero diagonal and all off-diagonal entries $\pm 1$, which satisfies $CC^\top = (n-1)I$. If $C$ is symmetric, then $C$ has eigenvalues $\pm\sqrt{n-1}$. Hence, its spectrum is symmetric. Conference matrices are well-studied; see for example \cite[Section 10.4]{brou-hamer}.
An important example of a symmetric conference matrix is the Seidel matrix of the Paley graph extended with an isolated vertex, where the {\em Paley graph} is defined on the elements of a finite field $\bf{F}_{q}$, with $q\equiv 1~(\mathrm{mod}~4)$, where two elements are adjacent whenever the difference is a nonzero square
in $\bf{F}_q$. The Paley graph is self-complementary. Therefore, by Proposition~\ref{cor:oddcase}, $C$ is the adjacency
matrix of a sign-symmetric complete signed graph. However, there exist many more symmetric conference matrices, including several that are not sign-symmetric (see \cite{BMS}).
\end{rem}

In \cite{belardo2019}, the authors posed the following problem on the existence of the non-complete signed graphs which are not sign-symmetric but have symmetric spectrum.

\begin{prob}[\cite{belardo2019}]
Are there non-complete connected signed graphs whose spectrum is symmetric with respect to the origin but they are not sign-symmetric?
\end{prob}

We answer this problem by showing that there exists such a graph for any order $n\ge6$.
For $s\ge0$, define the signed graph $\g_s$ to be the graph illustrated in Fig.~\ref{fig:gs}.

\begin{figure}[h!]
\centering
\begin{tikzpicture}[scale=.7]
\vertex [fill] (a1) at (0,2) [label=left:$1$]{};
\vertex [fill](a2) at (0,0) [label=left:$2$]{};
\vertex [fill](a3) at (1.77,-1) [label=below:$3$]{};
\vertex [fill](a4) at (3.54,0) [label=right:$4$]{};
\vertex [fill](a5) at (3.54,2) [label=right:$5$]{};
\vertex [fill](a6) at (1.77,3) [label=above:$6$]{};
\vertex [fill](b1) at (1.77,4.6) []{};
\vertex [fill](b2) at (1.77,5.8) []{};
\draw [fill=black] (1.77,5.45) circle (.5pt);
\draw [fill=black] (1.77,5.2) circle (.5pt);
\draw [fill=black] (1.77,4.95) circle (.5pt);
\draw [dotted,rotate around={-90:(1.77,5.2)},line width=.5pt] (1.77,5.2) ellipse (1.15cm and 0.4cm);
\draw (3.45,5.5) node {$s$\,vertices};
\draw [dashed] (0,2)-- (3.54,0);
\path
(a1) edge (a2)
(a2) edge (a3)
(a3) edge (a4)
(a4) edge (a5)
(a5) edge (a6)
(a1) edge (a6)
(a4) edge (a6)
(b1) edge (a1)
(b1) edge (a5)
(b2) edge (a1)
(b2) edge (a5);
\end{tikzpicture}
 \caption{The graph $\g_s$ }\label{fig:gs}
\end{figure}
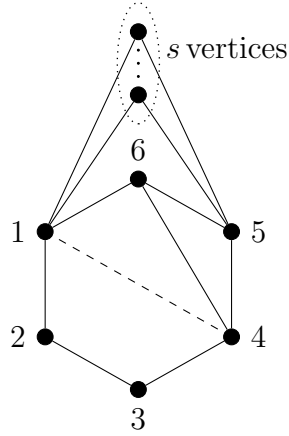
\begin{thm}\label{thm:problem} For $s\ge0$, the graph $\g_s$ has a symmetric spectrum, but it is not sign-symmetric.
\end{thm}
\begin{proof} Let $S$ be the set of $s$ vertices adjacent to both $1$ and $5$.
The  positive $5$-cycles of $\g_s$ are $123461$ together with $u1645u$ for any $u\in S$, and the negative $5$-cycles  are  $u1465u$ for any $u\in S$.
Hence, $c_5^+=s+1$ and $c_5^-=s$. In view of Observation~\ref{obs:sign-sym}, this shows that $\g_s$ is not sign-symmetric.

Next, we show that $\g_s$ has a symmetric spectrum. It suffices to verify that $a_{2k+1}=0$ for $k=1, 2, \ldots$.

The graph $\g_s$ contains a unique positive cycle of length $3$: $4564$ and a unique negative cycle of length $3$: $1461$. It follows that $a_3=0$.

As discussed above, we have $c_5^+=s+1$ and $c_5^-=s$.
We count the number of positive and negative copies of $K_2\cup C_3$.
For the negative triangle $1461$, there are $s+1$ non-incident edges, namely $23$ and $5u$ for any $u\in S$ and for the positive triangle $4564$, there are $s+2$ non-incident edges, namely $12$, $23$ and $1u$ for any $u\in S$.
It follows that
$$a_5=-2((s+1)-s)+2((s+2)-(s+1)=0.$$

Now, we count the number of positive and negative elementary subgraphs on $7$ vertices:
\begin{itemize}
  \item[$C_7$:]  ~$s$ positive:  $u123465u$ for any $u\in S$, and no negative;
  \item[$K_2\cup C_5$:]  ~$2s$ positive: $u5\cup 123461$, and $23 \cup u1645u$  for any $u\in S$, and $s$ negative: $23 \cup u1465u$ for any $u\in S$;
  \item[$2K_2\cup C_3$:] ~$s+1$ positive: $u1\cup 23\cup 4564$ for any $u\in S$, and $s+1$ negative: $u5\cup 23\cup 1461$ for any $u\in S$;
  \item [$C_4\cup C_3$:]  ~none.
\end{itemize}
Therefore,
$$a_7=-2(s-0)+2(2s-s)-2((s+1)-(s+1))=0.$$
The graph $\g_s$ contains  no elementary subgraph on $8$ vertices or more. The result now follows.
\end{proof}

More families of non-complete signed graphs with a symmetric spectrum but not sign-symmetric can be found.
Consider the signed graphs $\g_{s,t}$ depicted in Fig.~\ref{fig:gs,t}, in which the number of upper repeated pair of vertices is $s\ge0$ and the number of upper repeated pair of vertices is $t\ge1$. In a similar fashion as in the proof of Theorem~\ref{thm:problem} it can be verified that $\g_{s,t}$ has a symmetric spectrum, but it is not sign-symmetric.
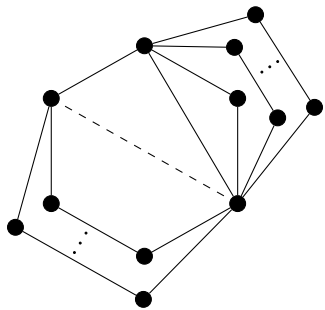
\begin{figure}[h!]
\centering
\begin{tikzpicture}[scale=.7]
\vertex[fill] (a1) at (0,2) []{};
\vertex [fill](a2) at (0,0) []{};
\vertex [fill](a3) at (1.77,-1) []{};
\vertex [fill](a4) at (3.54,0) []{};
\vertex [fill](a5) at (3.54,2) []{};
\vertex [fill](a6) at (1.77,3) []{};
\vertex [fill](b2) at (-.68,-.45) []{};
\vertex [fill](b3) at (1.75,-1.82) []{};
\vertex [fill](b4) at (4.3,1.63) []{};
\vertex [fill](b44) at (5,1.83) []{};
\vertex [fill](b6) at (3.48,2.97) []{};
\vertex [fill](b66) at (3.88,3.59) []{};
\draw [fill=black] (.66,-.56) circle (.5pt);
\draw [fill=black] (.55,-.75) circle (.5pt);
\draw [fill=black] (.44,-.93) circle (.5pt);
\draw [fill=black] (4,2.5) circle (.5pt);
\draw [fill=black] (4.15,2.6) circle (.5pt);
\draw [fill=black] (4.3,2.7) circle (.5pt);
\draw [dashed] (0,2)-- (3.54,0);
\path
(a1) edge (a2)
(a2) edge (a3)
(a3) edge (a4)
(a4) edge (a5)
(a5) edge (a6)
(a1) edge (a6)
(a4) edge (a6)
(b2) edge (b3)
(b2) edge (a1)
(b3) edge (a4)
(b4) edge (b6)
(b4) edge (a4)
(b6) edge (a6)
(b44) edge (b66)
(b44) edge (a4)
(b66) edge (a6)
;
\end{tikzpicture}
 \caption{The family of signed graphs $\g_{s,t}$}\label{fig:gs,t}
\end{figure}



\begin{figure}[!hb]
\minipage{1\textwidth}
\includegraphics[width=\linewidth]{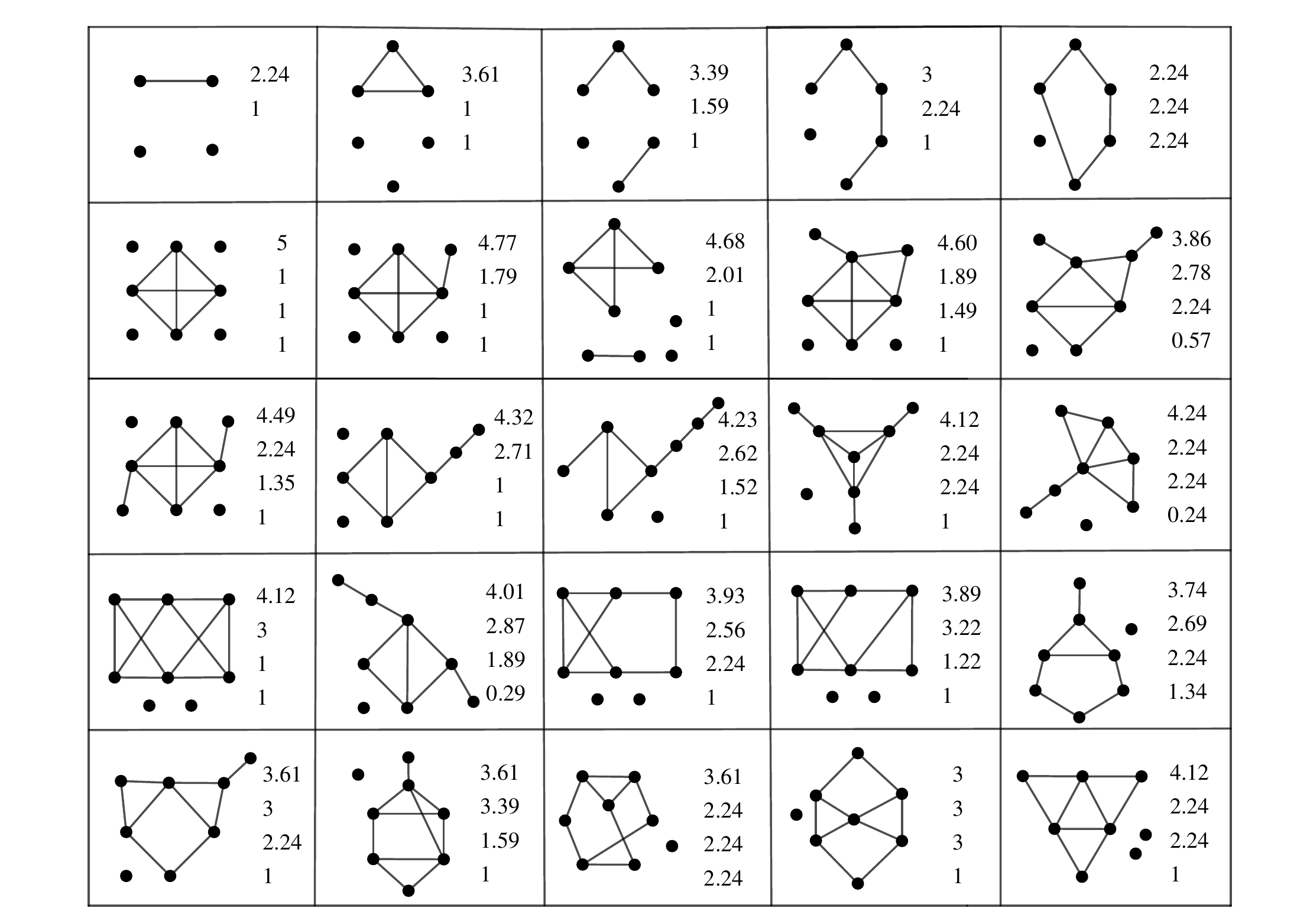}
\caption{Complete signed graphs (up to switching isomorphism and negation) of order $4, 6, 8$ having symmetric spectrum. The numbers next to the graphs are the non-negative eigenvalues. Only the last graph on the right is not sign-symmetric.}\label{sign4,6,8}
\endminipage
\end{figure}

\newpage

\begin{figure}[!htb]
\minipage{1\textwidth}
\includegraphics[width=\linewidth]{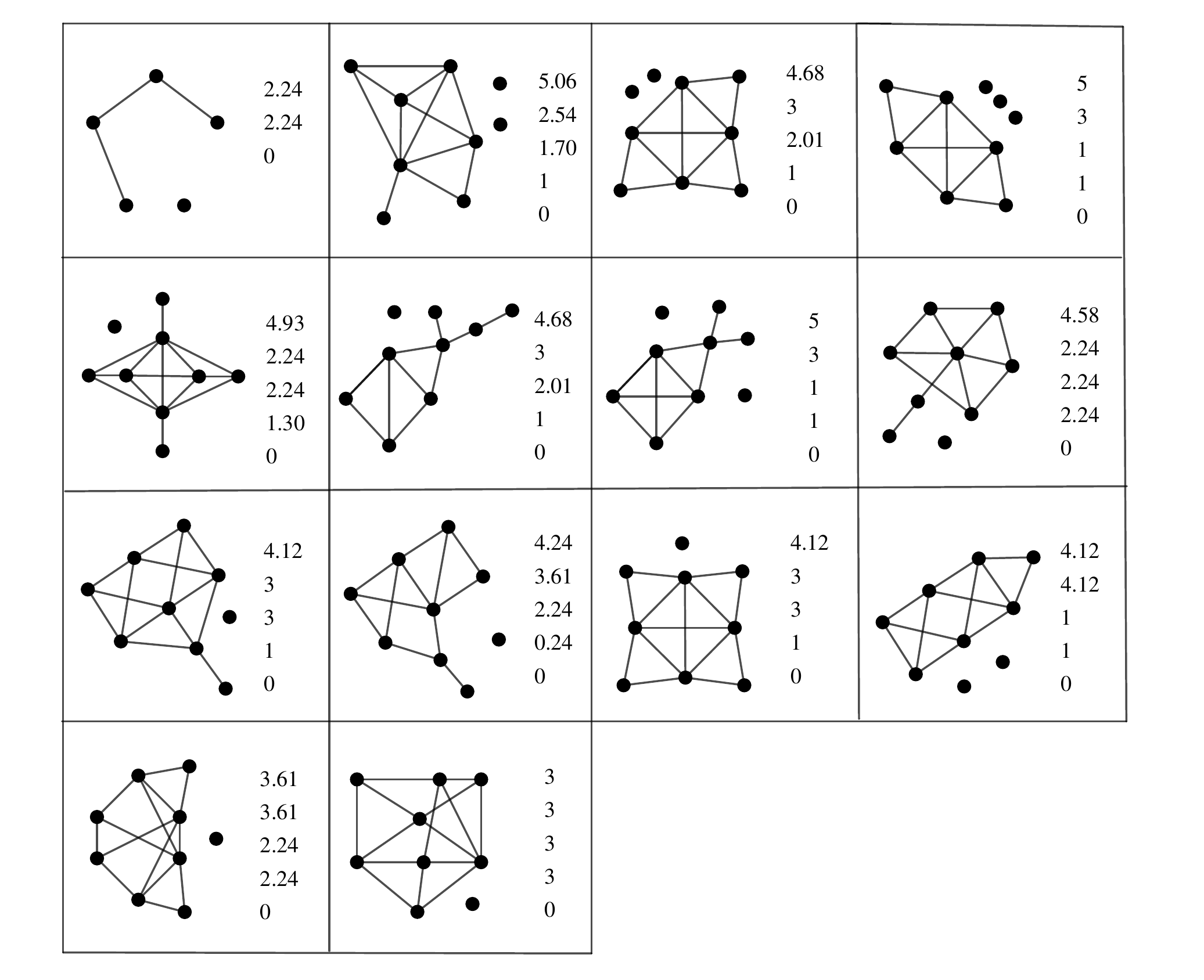}
\caption{Complete signed graphs (up to switching isomorphism and negation) of order $5, 9$ having symmetric spectrum. The numbers next to the graphs are the non-negative eigenvalues.The first three signed graphs of order $9$ are not sign-symmetric.}\label{sign5,9}
\endminipage
\end{figure}

\end{document}